\documentclass{article}
\usepackage{amsmath}
\usepackage[dvipdfm]{graphicx,color}
\usepackage{amsthm}
\usepackage{amssymb}
\usepackage{amscd}
\usepackage{fancybox}
\usepackage{ascmac}

\usepackage{slashbox}
\usepackage{multirow}

\newtheorem{thm}{Theorem}[section]
\newtheorem{lem}{Lemma}[section]

\theoremstyle{remark}
\newtheorem{remark}{Remark}[section]

\begin{document}
\title{Upper bounds for the minimal number of singular fibers in a Lefschetz fibration over the torus}
\author{Noriyuki Hamada}
\date{}
\maketitle

\begin{abstract}
In this paper, we give some relations in the mapping class groups of oriented closed surfaces 
in the form that a product of a small number of right hand Dehn twists is equal to a single commutator.
Consequently, we find upper bounds for the minimal number of singular fibers in a Lefschetz fibration over the torus.
\end{abstract}

\section{Introduction} \label{sec:intro}

Lefschetz fibrations were originally introduced for studying topological
properties of smooth complex projective varieties, and afterwards
generalized to differentiable category. Furthermore Donaldson and Gompf
revealed the close relationship between Lefschetz fibrations and
$4$-dimensional symplectic topology in the late 1990s, and since then
they have been extensively studied.

The information about the number of singular fibers in a Lefschetz fibration provides us 
important information about 
the topological invariants of its total space such as the Euler number, the signature, the Chern numbers, and so on.
In addition, it has been known that the number of singular fibers in a Lefschetz fibration cannot be arbitrary, 
so it makes sense to ask what the minimal number of singular fibers in a Lefschetz fibration is.
We denote by $N(g,h)$ the minimal number of singular fibers 
in a non-trivial relatively minimal genus $g$ Lefschetz fibration
over the oriented closed surface of genus $h$.
This minimal number has been studied by various authors.
Table~\ref{tbl:N(g,h)} shows previous studies about $N(g,h)$.
\begin{table}[t]
\begin{center}
\begin{tabular}{c||c|c|c|c|c|c|c} 
\rule[-6pt]{0pt}{18pt} $7$ & $6 \leq N \leq 24$ & $2 \leq N \leq 24$ & $1$ & $1$ & $1$ & $1$ & $1$ \\ \hline
\rule[-6pt]{0pt}{18pt} $6$ & $6 \leq N \leq 16$ & $2 \leq N \leq 16$ & $1$ & $1$ & $1$ & $1$ & $1$ \\ \hline
\rule[-6pt]{0pt}{18pt} $5$ & $5 \leq N \leq 20$ & $2 \leq N \leq 20$ & $1$ & $1$ & $1$ & $1$ & $1$ \\ \hline
\rule[-6pt]{0pt}{18pt} $4$ & $4 \leq N \leq 12$ & $2 \leq N \leq 12$ & $1$ & $1$ & $1$ & $1$ & $1$ \\ \hline
\rule[-6pt]{0pt}{18pt} $3$ & $3 \leq N \leq 16$ & $2 \leq N \leq 16$ & $1$ & $1$ & $1$ & $1$ & $1$ \\ \hline
\rule[-6pt]{0pt}{18pt} $2$ & $7$ & $6$ or $7$  & $5$ or $6$ & $5$ & $5$ & $5$ & $5$ \\ \hline
\rule[-6pt]{0pt}{18pt} $1$ & $12$ &  $12$ & $12$ & $12$ & $12$ & $12$ & $12$ \\ \hline \hline
\slashbox[8pt]{$g$}{$h$} & $0$ & $1$  & $2$ & $3$ & $4$ & $5$ & $6$
\end{tabular}
\end{center}
  \caption{Previous results for $N = N(g,h)$}
  \label{tbl:N(g,h)}
\end{table}
Korkmaz and Ozbagci proved that $(1)$ $N(g,h)=1$ if and only if $g \geq 3$ and $h \geq 2$, $(2)$ $N(1,h) = 12$ for all $h \geq 0$, and 
$(3)$ $5 \leq N(2,h) \leq 8$ for all $h \geq 0$ (\cite{KO1}).
The upper bound for $N(2,h)$ in $(3)$ follows from the existence of a genus $2$ Lefschetz fibration over the sphere 
with eight singular fibers, which was constructed by Matsumoto (\cite{Matsumoto}).
In addition, for $g = 2$ Monden showed that $(1)$ $N(2,h) = 5$ for all $h \geq 3$, $(2)$ $N(2,2) \leq 6$, and 
$(3)$ $6 \leq N(2,1) \leq 7$ (\cite{Mon}).
Ozbagci proved that the number of singular fibers in a genus $2$ Lefschetz fibration over the sphere
cannot be equal to $5$ or $6$ (\cite{Oz}) and 
Xiao constructed a genus $2$ Lefschetz fibration over the sphere with seven singular fibers (\cite{Xiao}); hence, 
$N(2,0) = 7$.
For $h=0$ some estimates for $N(g,0)$ are known. 
Cadavid and Korkmaz independently generalized Matsumoto's genus $2$ Lefschetz fibration as above 
to genus $g$ Lefschetz fibrations over the sphere with $2g+10$ singular fibers
for $g$ odd, or with $2g+4$ singular fibers for $g$ even (\cite{C,K2}).
This fact shows $N(g,0) \leq 2g+10$ for $g$ odd and $N(g,0) \leq 2g+4$ for $g$ even.
Stipsicz (\cite{Sti}) proved that for any $g \geq 2$, the number of irreducible singular fibers 
in a genus $g$ Lefschetz fibration over the sphere is bounded below by $(4g+2)/5$.
Therefore, we have
$$
\frac{1}{5}(4g+2) \leq N(g,0).
$$
In particular, there is no universal upper bound for $N(g,0)$ which is independent of $g$.

In the case of $h=1$, it has been known that $2 \leq N(g,1) \leq N(g,0)$.
The former inequality follows by \cite{KO1}.
The latter inequality comes from the following observation.
If a genus $g$ Lefschetz fibration over the sphere is given, then by taking the fiber sum of it and $\Sigma_g \times \Sigma_1$ 
the trivial $\Sigma_g$-bundle over the torus, we can construct 
a genus $g$ Lefschetz fibration over the torus without changing the number of the singular fibers.
However, non-trivial upper bounds for $N(g,1)$ have not been given explicitly 
as far as the author knows.
The present paper aims at giving new upper bounds for $N(g,1)$, 
i.e.\ in the case of Lefschetz fibrations over the torus.
The main theorem of this paper is the following.

\begin{thm} \label{mainthm} 
For the minimal number of singular fibers in a Lefschetz fibration over the torus, the following holds:
\begin{flushleft}
\textup{(1)} \quad $N(g,1) \leq 6 \qquad \text{ for all } g \geq 3.$ \\
\textup{(2)} \quad $N(g,1) \leq 5 \qquad \text{ for all } g \geq 7.$ 
\end{flushleft}
In particular, there is a universal upper bound for $N(g,1)$ which does not depend on $g$.
\end{thm}


We will prove Theorem~\ref{mainthm} by concretely constructing 
a genus $g$ Lefschetz fibration over the torus
with six or five singular fibers for arbitrary $g \geq 3$ or $g \geq 7$, respectively.
This will be done by providing new relations in the mapping class group of the surface of genus $g$, 
which are in the form that a product of five or six right hand Dehn twists is equal to a single commutator.

The contents of this paper are as follows.
Section $\ref{sec:basis}$ consists of the fundamental concepts on Lefschetz fibrations.
In particular, we will illustrate the fact that 
Lefschetz fibrations can be obtained from a certain type of relations 
in the mapping class groups of fiber surfaces.
In Section $\ref{sec:relations}$ we will introduce Matsumoto's relation and 
the $7$-holed torus relation
in the mapping class groups of holed surfaces, which will be used in the proof of Theorem~\ref{mainthm}.
Finally, in Section $\ref{sec:proof}$ we will prove Theorem~\ref{mainthm}.
We will regard Matsumoto's relation and the $7$-holed torus relation 
as relations in the mapping class groups of other closed surfaces 
by embedding the original holed surfaces into the closed surfaces.
Then we will transform the relations to new ones, which will prove Theorem~\ref{mainthm}.\\

\noindent\textbf{Acknowledgement.}
The author would like to thank Hisaaki Endo, Kenta Hayano, Susumu Hirose, Yukio Matsumoto,
Naoyuki Monden and Yoshihisa Sato
for helpful comments and invaluable advice on Lefschetz fibrations or on the surface mapping class groups.
In particular, the author would like to thank Naoyuki Monden for several discussions on $N(g,h)$.
Finally, the author would like to express his deepest gratitude to his supervisor Osamu Saeki 
for his constant encouragement and many useful discussions.

\section{Lefschetz fibrations} \label{sec:basis}
We recall some basic definitions and facts about Lefschetz fibrations (for details, see~\cite{GS}).
Let $X^4$ be a connected, oriented and closed smooth 4-dimensional manifold, and $\Sigma_g$ be the connected, oriented and closed smooth 2-dimensional manifold with genus $g$.
A smooth map $f : X^4 \rightarrow \Sigma_h$ is called a \textit{Lefschetz fibration over $\Sigma_h$} 
if $f$ has only finitely many critical points $p_1, p_2, \dots, p_n$, 
around each of which $f$ is expressed as $(z_1, z_2) \mapsto z_1^2 + z_2^2$ by local complex coordinates compatible with the orientations of the manifolds.
We assume the critical values $b_i = f(p_i)$ are distinct.
The inverse image of a regular value (or a critical value) is called a \textit{regular fiber} (resp.\ a \textit{singular fiber}).
We will also assume that the regular fibers are connected.
Since $f$ is a submersion on the complement of singular fibers, 
the restriction $f : X^4 \setminus ( f^{-1}(b_1) \cup f^{-1}(b_2) \cup \dots \cup f^{-1}(b_n) ) \rightarrow \Sigma_h \setminus \{ b_1, b_2, \dots, b_n \}$ 
is a $\Sigma_g$-bundle over $\Sigma_h \setminus \{ b_1, b_2, \dots, b_n \}$ for some $\Sigma_g$.
The $genus$ of the Lefschetz fibration is defined to be the genus of a regular fiber.
Furthermore, in this paper, we assume the relative minimality and the non-triviality.
A Lefschetz fibration is said to be \textit{relatively minimal} if there is no fiber 
which contains a $(-1)$-sphere (embedded sphere with self-intersection $-1$), 
and \textit{non-trivial} if it has at least one singular fiber.
The singular fiber $f^{-1}(b_i)$ is obtained by ``crushing'' a simple closed curve $c_i$, called the \textit{vanishing cycle}, on a nearby regular fiber to a point.
If the vanishing cycle is separating (or non-separating), then the corresponding singular fiber is said to be \textit{reducible} (resp.\ \textit{irreducible}).
Two Lefschetz fibrations $f:X \rightarrow \Sigma_h$ and $f' : X' \rightarrow \Sigma_h$ are said to be \textit{isomorphic}
if there are orientation-preserving diffeomorphisms 
$\Psi : X \rightarrow X'$ and $\psi : \Sigma_h \rightarrow \Sigma_h$ 
such that $f' \circ \Psi = \psi \circ f$.

There is a good relation between Lefschetz fibrations and the surface mapping class groups (for details, see~\cite{Matsumoto}).
We fix a regular value $b_0 \in \Sigma_h$ and an identification $\iota$ between the regular fiber $f^{-1}(b_0)$ and the model surface $\Sigma_g$.
Let $\gamma$ be a smooth loop in $\Sigma_h$ based at $b_0$. 
Then the pull-back bundle $\gamma^*(f)$ is described as $\Sigma_g \times [0,1]$ 
with $\Sigma_g \times 0$ and $\Sigma_g \times 1$ being identified via an orientation-preserving diffeomorphism 
$\phi$ from $\Sigma_g$ to itself: 
$f^{-1}(\gamma) \cong \Sigma_g \times [0,1] \hspace{2pt}\slash\hspace{2pt} (x,1) \sim (\phi(x),0)$.
Let $\mathcal{M}_g$ be the mapping class group of genus $g$ which consists of 
all isotopy classes of orientation-preserving diffeomorphisms of $\Sigma_g$.
Then the map $ \Phi : \pi_1 ( \Sigma_g \setminus \{ b_1, b_2, \dots , b_n \} , b_0) \rightarrow \mathcal{M}_g$ 
which maps $\gamma$ to $\phi$ is well-defined, and becomes an anti-homomorphism.
Here we do not distinguish a curve or a diffeomorphism from its homotopy or isotopy class, respectively, and we will use this convention throughout this paper.
We call the map $\Phi$ the \textit{monodromy representation} of the Lefschetz fibration $f$.
Let $\gamma_i$ be a loop on $\Sigma_g \setminus \{ b_1, b_2, \dots, b_n \} $ 
based at $b_0$ which surrounds exclusively $b_i$ with the orientation as depicted in Figure~\ref{fig:gammaionsigma_h}.
\begin{figure}[t]
 \begin{center}
  \includegraphics[keepaspectratio, width=98pt, clip]{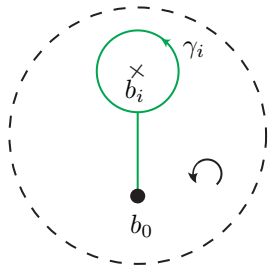}
    \put(-52,12){$b_0$}
    \put(-54,63){$b_i$}
    \put(-32,80){$\gamma_i$} 
  \caption{The oriented loop $\gamma_i$ on $\Sigma_h$}
  \label{fig:gammaionsigma_h}
 \end{center}
\end{figure}
The monodromy representation $\Phi$ maps $\gamma_i$ to the right hand Dehn twist $t_{c_i}$ along the corresponding vanishing cycle $c_i$:
$\Phi(\gamma_i) = t_{c_i}$ (see Figure~\ref{fig:righthandDehntwist}).
\begin{figure}[t]
 \begin{center}
  \includegraphics[keepaspectratio, width=230pt, clip]{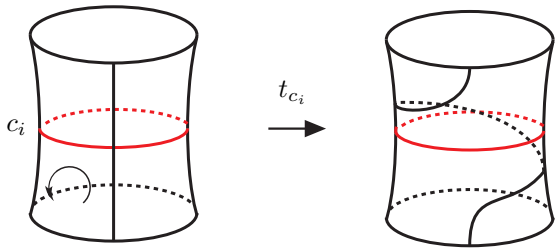}
    \put(-220,57){$c_i$} 
    \put(-117,70){$t_{c_i}$}
  \caption{Right hand Dehn twist $t_{c_i}$ along the simple closed curve $c_i$}
  \label{fig:righthandDehntwist}
 \end{center}
\end{figure}
If we change the identification $\iota : f^{-1}(b_0) \rightarrow \Sigma_g$ to another one, 
then the monodromy representation $\Phi$ changes to $\rho \Phi \rho^{-1}$ for some $\rho \in \mathcal{M}_g$.
A Lefschetz fibration determines the monodromy representation up to such a conjugation.
Conversely, if an anti-homomorphism $\Phi : \pi_1 (\Sigma_h \setminus \{b_1, b_2, \dots, b_n \}, b_0 ) \to \mathcal{M}_g $ 
which maps each $\gamma_i$ to a right hand Dehn twist is given, 
then one can construct a relatively minimal genus $g$ Lefschetz fibration over $\Sigma_h$ with its monodromy representation $\Phi$. 
Moreover, if $g \geq 2$, then such a Lefschetz fibration is determined uniquely up to an isomorphism. 

Furthermore, Lefschetz fibrations correspond to a certain type of relations in the mapping class groups.
The fundamental group $\pi_1(\Sigma_h \setminus \{b_1, b_2, \dots, b_n \}, b_0 )$ has the finite presentation 
\begin{center}
 $\pi_1 (\Sigma_h \setminus \{b_1, b_2, \dots, b_n \}, b_0 ) = 
\Bigl< \alpha_j,\beta_j,\gamma_i ~ \Bigl| ~ \prod\limits_{j=1}^h \bigl[ \alpha_j, \beta_j \bigl] = \prod\limits_{i=1}^n \gamma_i \Bigl>$,
\end{center}
where $\alpha_j$, $\beta_j$ ($j = 1, 2, \dots, h$) and $\gamma_i$ ($i = 1, 2, \dots, n$) are the loops as indicated in Figure \ref{fig:generatorforpi1}, and $\bigl[ \alpha_j, \beta_j \bigl] = \alpha_j\beta_j\alpha_j^{-1}\beta_j^{-1}$ 
represents the commutator of $\alpha_j$ and $\beta_j$.
\begin{figure}[t]
 \begin{center}
  \includegraphics[keepaspectratio, width=337pt, clip]{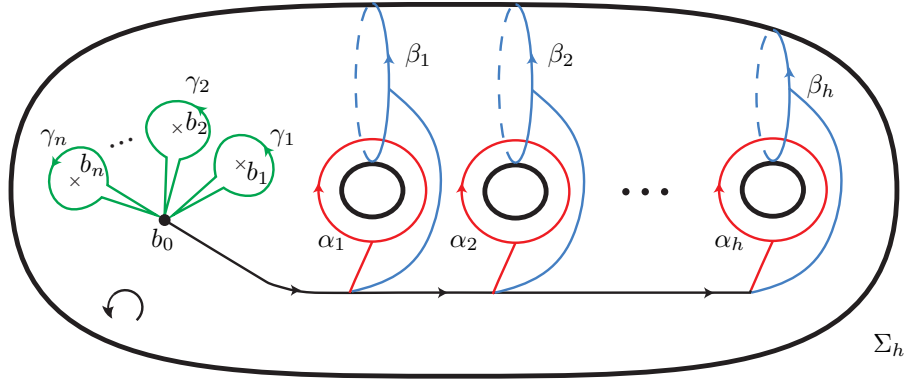}
    \put(-283,50){$b_0$}
    \put(-238,90){$\gamma_1$} 
    \put(-270,110){$\gamma_2$}
    \put(-325,90){$\gamma_n$}
    \put(-247,75){$b_1$} 
    \put(-270,94){$b_2$}
    \put(-310,78){$b_n$}
    \put(-220,50){$\alpha_1$}
    \put(-187,120){$\beta_1$}
    \put(-170,50){$\alpha_2$}
    \put(-133,120){$\beta_2$}
    \put(-70,50){$\alpha_h$}
    \put(-35,108){$\beta_h$}
    \put(-10,10){$\Sigma_h$}
  \caption{The generators of $\pi_1(\Sigma_h \setminus \{b_1, b_2, \dots, b_n \}, b_0 )$}
  \label{fig:generatorforpi1}
 \end{center}
\end{figure}
Thus, a monodromy representation $\Phi$ satisfies $\prod_j [\Phi(\alpha_j),\Phi(\beta_j)] = \prod_i t_{c_i}$
and $t_{c_i} = \Phi(\gamma_i)$ are right hand Dehn twists.
We call this relation the \textit{global monodromy} or the \textit{monodromy factorization} 
of the Lefschetz fibration corresponding to $\Phi$.
Conversely, if there is a relation in the form that ``a product of $n$ right hand Dehn twists is equal to
a product of $h$ commutators in $\mathcal{M}_g$",
$$  \prod\limits_{j=1}^h \bigl[ \phi_j, \psi_j \bigl] = \prod\limits_{i=1}^n t_{c_i}, $$
then we can define the anti-homomorphism 
$\Phi : \pi_1 (\Sigma_h \setminus \{b_1, b_2, \dots, b_n \}, b_0 ) \to \mathcal{M}_g $ 
by setting 
$\Phi(\alpha_j) = \phi_j$, $\Phi(\beta_j) = \psi_j$ and $\Phi(\gamma_i) = t_{c_i}$.
Consequently, we can construct a genus $g$ Lefschetz fibration over $\Sigma_h$ with $n$ singular fibers
such that its vanishing cycles are the simple closed curves ${c_i}$.


\section{Mapping class groups of holed surfaces} \label{sec:relations}
 \begin{figure}[t]
 \begin{center}
  \includegraphics[keepaspectratio, width=203pt, clip]{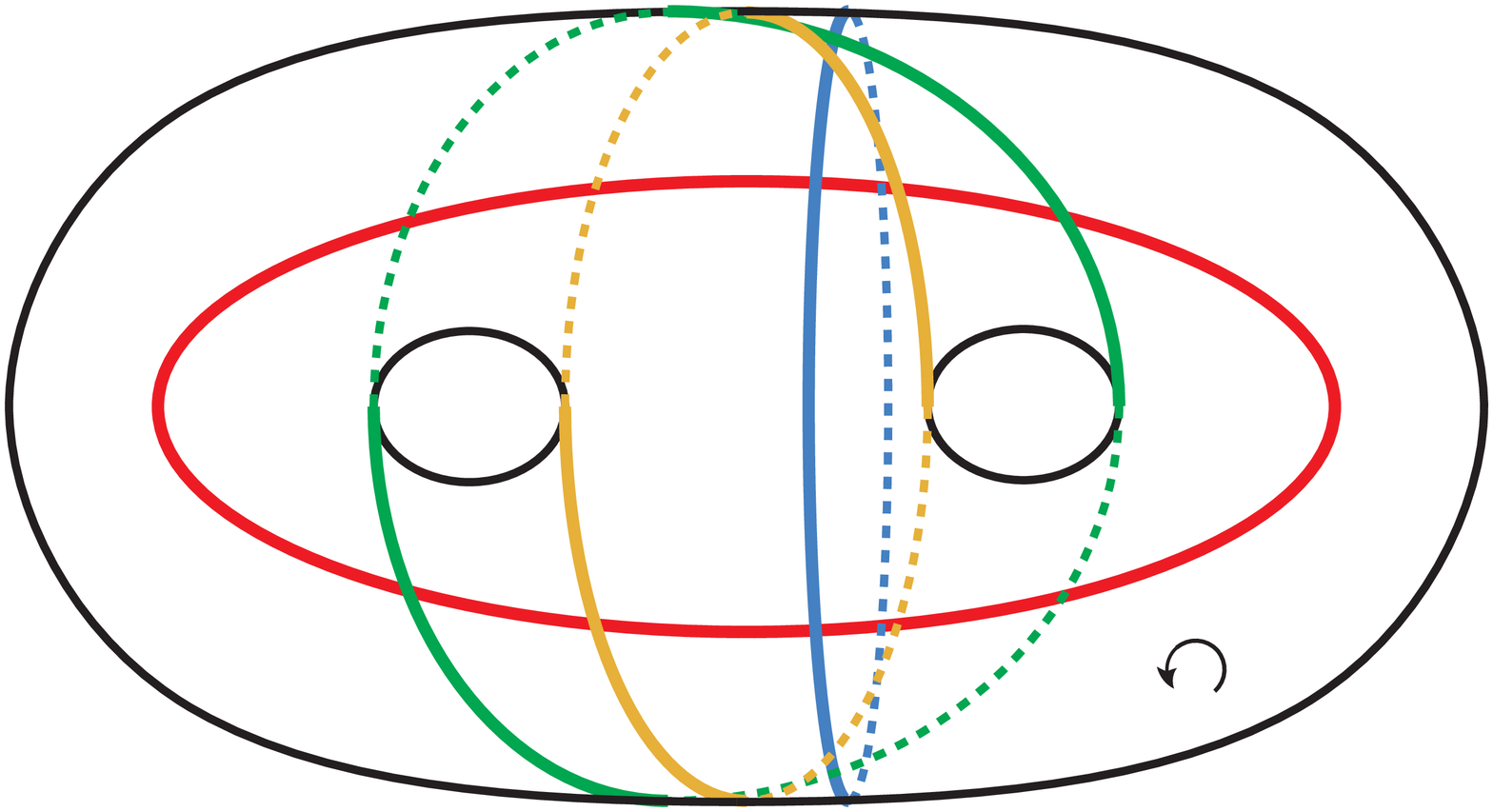}
       \put(-180,75){$B_0$} 
       \put(-150,10){$B_1$} 
       \put(-130,10){$B_2$} 
       \put(-105,50){$C$} 
       \put(-10,10){$\Sigma_2$}
     \caption{The curves for Matsumoto's relation}
     \label{fig:matsumotosrelation}
 \end{center}
\end{figure}
 \begin{figure}[t]
 \begin{center}
     \includegraphics[keepaspectratio, width=203pt,clip]{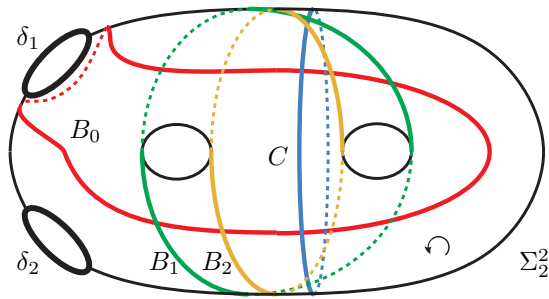}
       \put(-200,95){$\delta_1$} 
       \put(-200,10){$\delta_2$} 
       \put(-180,60){$B_0$} 
       \put(-150,10){$B_1$} 
       \put(-130,10){$B_2$} 
       \put(-105,50){$C$} 
       \put(-10,10){$\Sigma_2^2$}
     \caption{The curves for Matsumoto's relation with boundary}
     \label{fig:matsumotosrelationwithboundaries}
 \end{center}
\end{figure}

We will use the following notations:
\begin{itemize}
  \item[] $\Sigma = \Sigma_g^k$ : the compact oriented surface of genus $g$ with $k$ boundary components, 
  \item[] Diff$^+(\Sigma,\partial \Sigma)$: the group of orientation preserving self-diffeomorphisms of $\Sigma$ 
that are the identity on the boundary, 
  \item[] Diff$^+_0(\Sigma, \partial \Sigma)$: the normal subgroup of Diff$^+(\Sigma,\partial \Sigma)$ consisting of 
	all elements isotopic to the identity relative to the boundary, 
  \item[] $\mathcal{M}(\Sigma) = $Diff$^+(\Sigma,\partial \Sigma) / $Diff$^+_0(\Sigma, \partial \Sigma)$ : 
the mapping class group of $\Sigma$, 
  \item[] $\mathcal{M}_g^k = \mathcal{M}(\Sigma_g^k)$, \quad $\Sigma_g = \Sigma_g^0$, \quad $\mathcal{M}_g = \mathcal{M}_g^0$. 
\end{itemize}
We will use the functional notation for the product of $\mathcal{M}_g^k$, 
namely, for two elements $\phi$ and $\psi$ in $\mathcal{M}_g^k$ 
the product $\psi\phi$ means that we first apply $\phi$ tand hen $\psi$.
In the following, for a simple closed curve $\alpha$, the right hand Dehn twist $t_{\alpha}$ along $\alpha$ will 
often be written as $\alpha$, and the left hand Dehn twist $t_{\alpha}^{-1}$ as $\overline{\alpha}$, for simplicity.

\subsection{Matsumoto's relation}
Matsumoto~\cite{Matsumoto} constructed a genus $2$ Lefschetz fibration over the sphere on 
$(S^2 \times T^2) \sharp 4\overline{\mathbb{C}P_2}$ with eight singular fibers.
The global monodromy of the Lefschetz fibration is $$1 = (B_0 B_1 B_2 C)^2, $$
where the curves are as indicated in Figure \ref{fig:matsumotosrelation}.

In \cite{K} Korkmaz mentioned that the relation $$\delta_1 \delta_2 = (B_0 B_1 B_2 C)^2$$ 
holds in $\mathcal{M}_2^2$,
where $\delta_1$ and $\delta_2$ are as depicted in Figure \ref{fig:matsumotosrelationwithboundaries}, without proof.
For completeness, we give a proof here.
\begin{lem} \label{lem:matsumotosrelation} 
We have 
$\delta_1\delta_2 = (B_0 B_1 B_2 C)^2$
in $\mathcal{M}_2^2$.
\end{lem}
\begin{proof}
Here we use the symbol $t_{\alpha}$ for the right hand Dehn twist along a simple closed curve $\alpha$
in order to distinguish maps from curves.
We can determine whether two diffeomorphisms of a surface belong to the same mapping class or not 
by looking at their actions on a well-chosen collection of oriented curves which fill up the surface 
and which satisfies some technical assumptions 
(for details, see~\cite{FM} in which this method is called the \textit{Alexander method}).
We choose such a collection of curves $\{\gamma_i\}$ filling up $\Sigma_2^2$ as in Figure \ref{fig:fill-up-curves}.
If the actions of $t_{\delta_1}t_{\delta_2}$ and $(t_{B_0}t_{B_1}t_{B_2}t_C)^2$ on this system coincide, 
then they also coincide as mapping classes.
Indeed, their actions are identical as shown in Figures \ref{fig:gamma1} -- \ref{fig:gamma5}. 
(In the figures the orientations of the surfaces are the same as that indicated in  
Figure~\ref{fig:matsumotosrelationwithboundaries}.) \qedhere
\end{proof}
\begin{figure}[t]
 \begin{center}
  \includegraphics[keepaspectratio, width=220pt, clip]{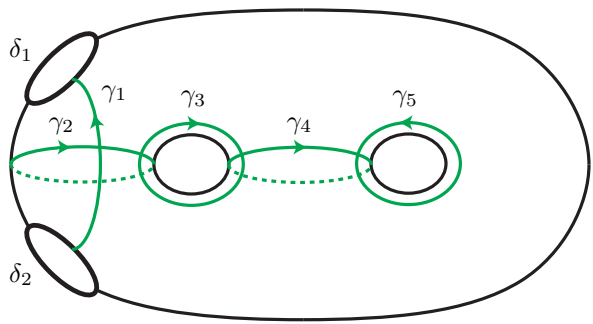}
  \put(-220,100){$\delta_1$} 
  \put(-220,15){$\delta_2$} 
  \put(-185,85){$\gamma_1$}
  \put(-205,73){$\gamma_2$}
  \put(-155,83){$\gamma_3$}
  \put(-115,73){$\gamma_4$}
  \put(-75,83){$\gamma_5$}
  \caption{The curves filling up $\Sigma_2^2$}
  \label{fig:fill-up-curves}
 \end{center}
\end{figure}
\begin{figure}[!htb]
 \begin{center}
  \includegraphics[keepaspectratio, width=364pt, clip]{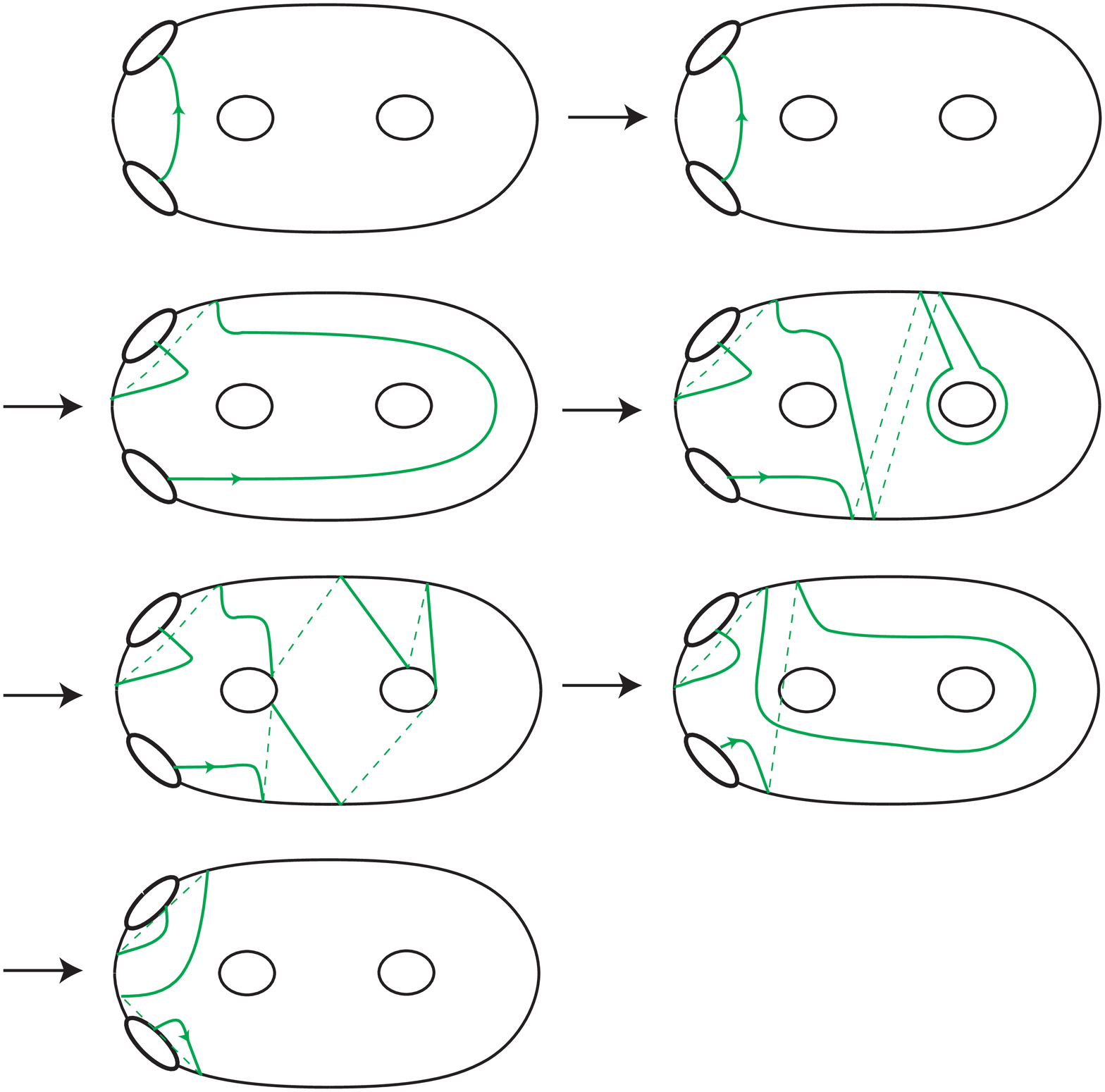}
  \put(-305,330){$\gamma_1$}
  \put(-185,335){$t_{B_1}t_{B_2}t_C$}
  \put(-360,238){$t_{B_0}$}
  \put(-175,238){$t_C$}
  \put(-360,140){$t_{B_2}$}
  \put(-175,142){$t_{B_1}$}
  \put(-360,50){$t_{B_0}$}
  \put(-300,15){$t_{\delta_1}t_{\delta_2}(\gamma_1)$}
  \caption{$t_{\delta_1}t_{\delta_2}(\gamma_1)=(t_{B_0}t_{B_1}t_{B_2}t_C)^2(\gamma_1)$}
  \label{fig:gamma1}
 \end{center}
\end{figure}
\begin{figure}[!htbp]
 \begin{center}
  \includegraphics[keepaspectratio, width=364pt, clip]{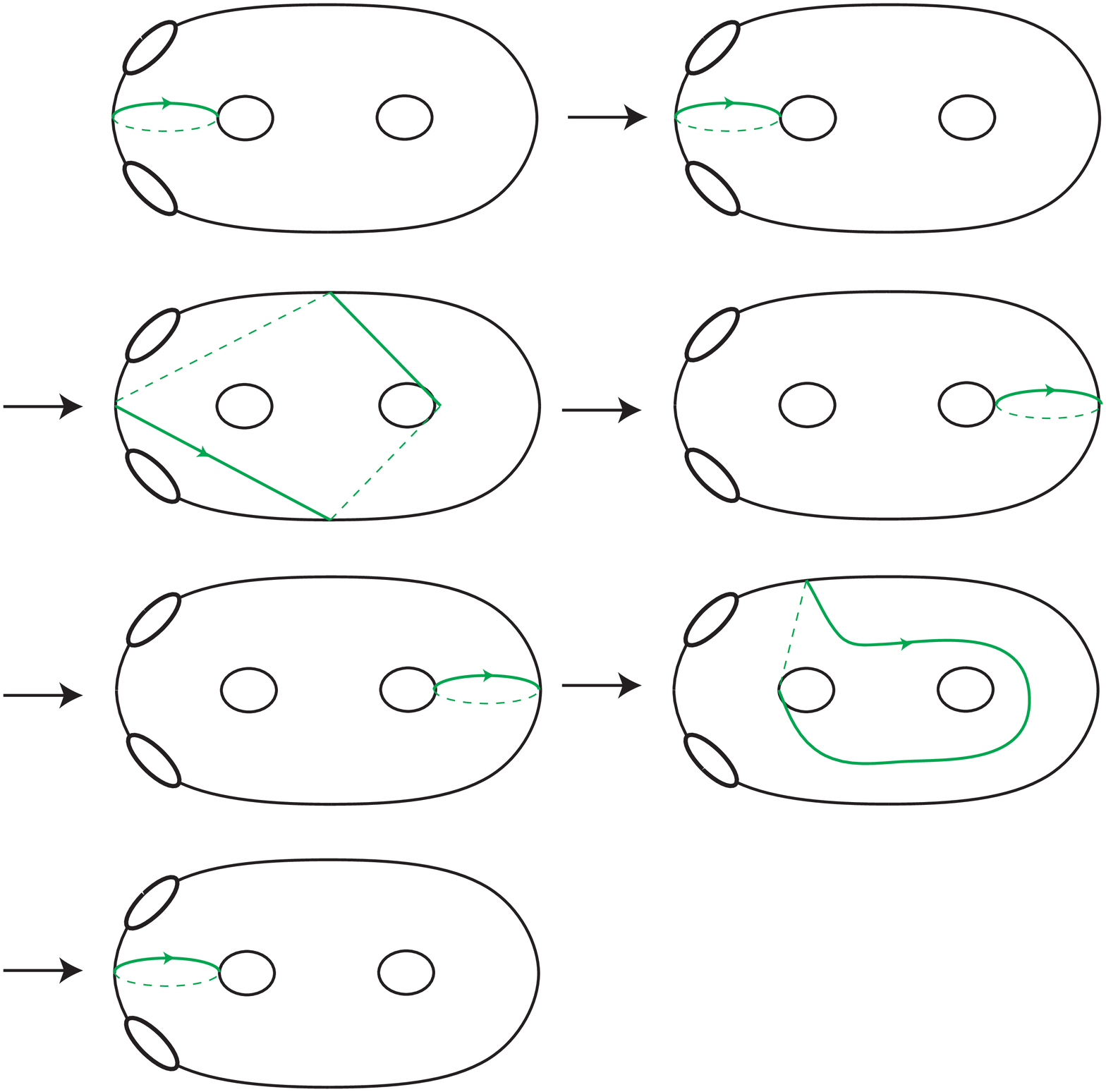}
  \put(-315,335){$\gamma_2$}
  \put(-180,335){$t_{B_2}t_C$}
  \put(-360,238){$t_{B_1}$}
  \put(-175,238){$t_{B_0}$}
  \put(-365,140){$t_{B_2}t_C$}
  \put(-175,142){$t_{B_1}$}
  \put(-360,50){$t_{B_0}$}
  \put(-310,50){$t_{\delta_1}t_{\delta_2}(\gamma_2)$}
  \caption{$t_{\delta_1}t_{\delta_2}(\gamma_2)=(t_{B_0}t_{B_1}t_{B_2}t_C)^2(\gamma_2)$}
  \label{fig:gamma2}
 \end{center}
\end{figure}
\begin{figure}[htbp]
 \begin{center}
  \includegraphics[keepaspectratio, width=364pt, clip]{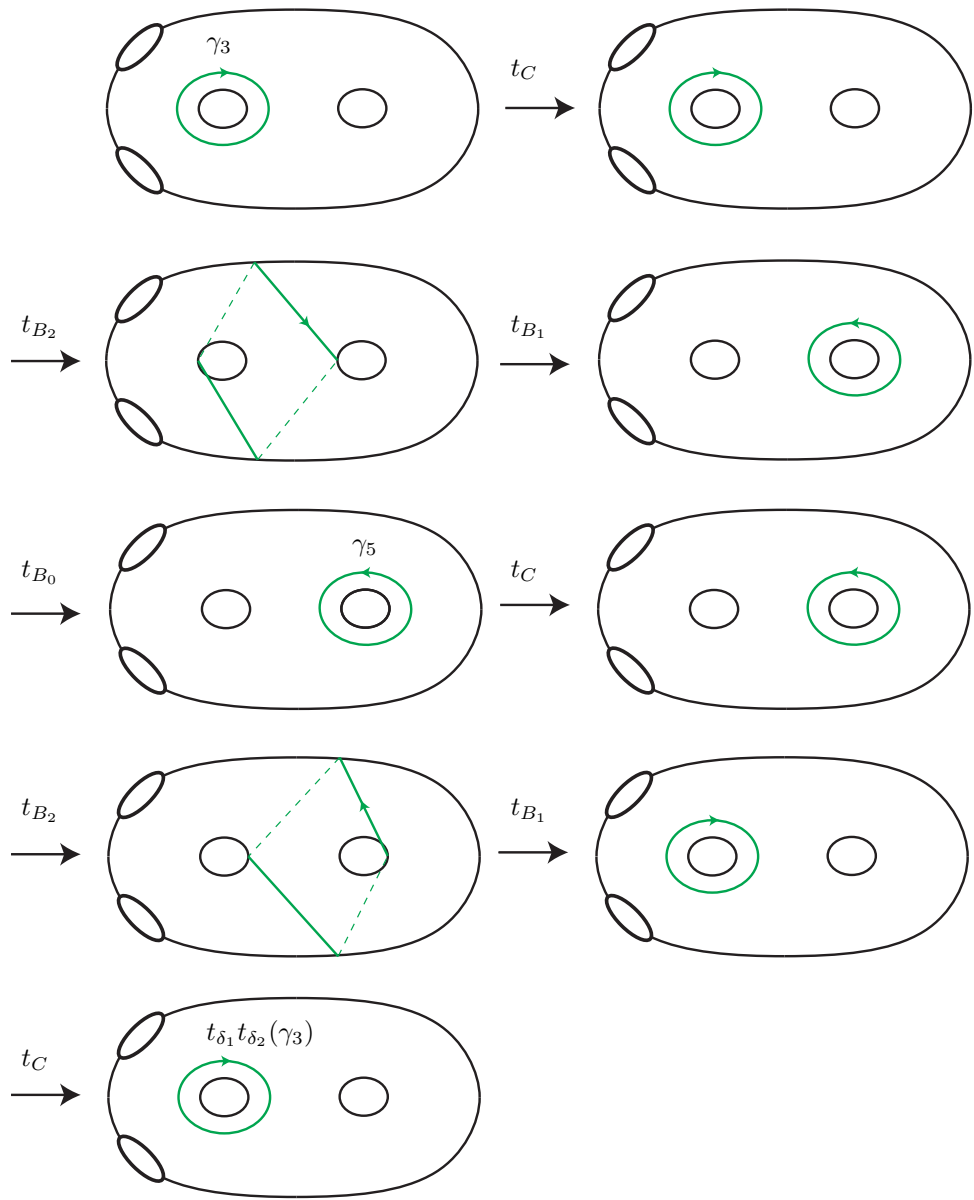}
  \put(-290,435){$\gamma_3$}
  \put(-175,425){$t_C$}
  \put(-360,328){$t_{B_2}$}
  \put(-175,328){$t_{B_1}$}
  \put(-360,235){$t_{B_0}$} 
  \put(-235,245){$\gamma_5$}
  \put(-175,235){$t_C$}
  \put(-360,145){$t_{B_2}$}
  \put(-175,145){$t_{B_1}$}
  \put(-360,50){$t_C$}
  \put(-290,60){$t_{\delta_1}t_{\delta_2}(\gamma_3)$}
  \caption{$t_{\delta_1}t_{\delta_2}(\gamma_3)=(t_{B_0}t_{B_1}t_{B_2}t_C)^2(\gamma_3)$. 
Note that $t_{B_0}t_{B_1}t_{B_2}t_C(\gamma_3)=\gamma_5$ and $t_{B_0}t_{B_1}t_{B_2}t_C(\gamma_5)=\gamma_3$.}
  \label{fig:gamma3}
 \end{center}
\end{figure}
\begin{figure}[htbp]
 \begin{center}
  \includegraphics[keepaspectratio, width=364pt, clip]{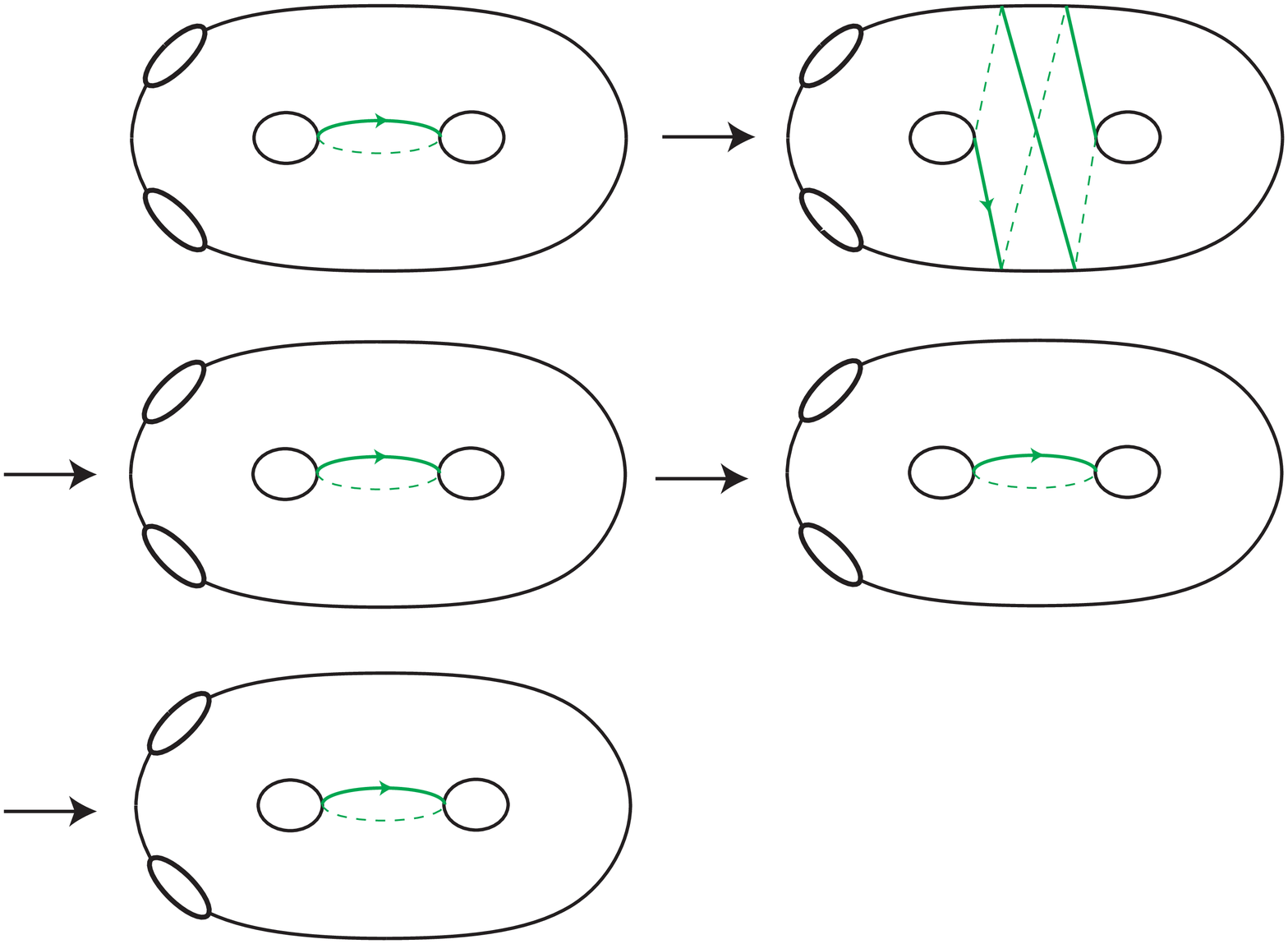}
  \put(-265,240){$\gamma_4$}
  \put(-175,238){$t_C$}
  \put(-360,143){$t_{B_2}$}
  \put(-180,142){$t_{B_0}t_{B_1}$}
  \put(-80,145){$\gamma_4$}
  \put(-385,50){$t_{B_0}t_{B_1}t_{B_2}t_C$}
  \put(-280,50){$t_{\delta_1}t_{\delta_2}(\gamma_4)$}
  \caption{$t_{\delta_1}t_{\delta_2}(\gamma_4)=(t_{B_0}t_{B_1}t_{B_2}t_C)^2(\gamma_4)$}
  \label{fig:gamma4}
 \end{center}
\end{figure}
\begin{figure}[htbp]
 \begin{center}
  \includegraphics[keepaspectratio, width=364pt, clip]{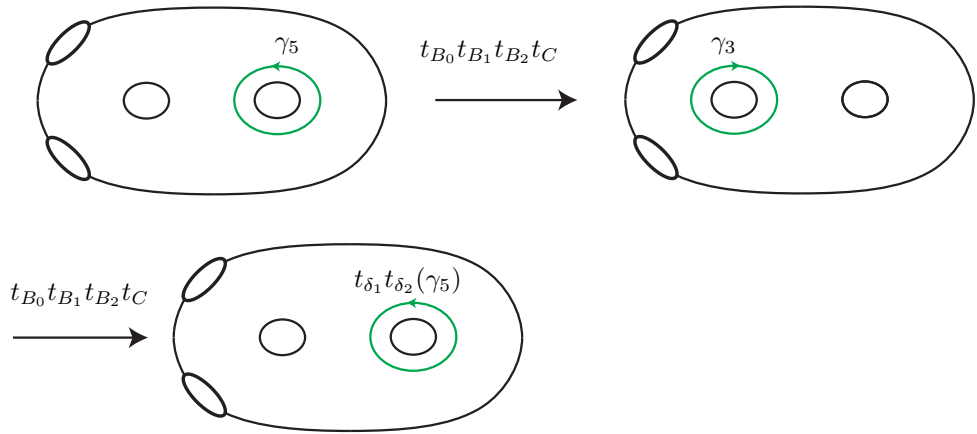}
  \put(-210,142){$t_{B_0}t_{B_1}t_{B_2}t_C$}
  \put(-265,145){$\gamma_5$}
  \put(-100,145){$\gamma_3$}
  \put(-365,50){$t_{B_0}t_{B_1}t_{B_2}t_C$}
  \put(-235,55){$t_{\delta_1}t_{\delta_2}(\gamma_5)$}
  \caption{$t_{\delta_1}t_{\delta_2}(\gamma_5)=(t_{B_0}t_{B_1}t_{B_2}t_C)^2(\gamma_5)$. Compare with Figure~\ref{fig:gamma3}.}
  \label{fig:gamma5}
 \end{center}
\end{figure}

\clearpage
\subsection{$k$-holed torus relations}
\begin{figure}[t]
 \begin{center}
     \includegraphics[keepaspectratio, width=97pt,clip]{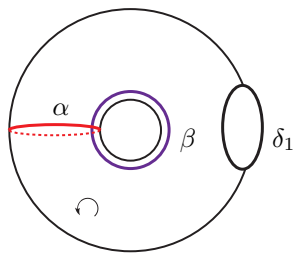}
       \put(3,40){$\delta_1$} 
       \put(-80,52){$\alpha$}
       \put(-32,40){$\beta$}
     \caption{$1$-holed torus}
     \label{fig:1-holedtorusrelation}
 \end{center}
\end{figure}
\begin{figure}[t]
 \begin{center}
     \includegraphics[keepaspectratio, width=97pt,clip]{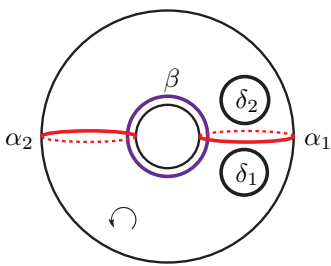}
       \put(-23,59){$\delta_2$} 
       \put(-23,31){$\delta_1$} 
       \put(3,45){$\alpha_1$}
       \put(-110,45){$\alpha_2$}
       \put(-50,67){$\beta$}
     \caption{$2$-holed torus}
     \label{fig:2-holedtorusrelation}
 \end{center}
\end{figure}
Korkmaz and Ozbagci~\cite{KO2} systematically constructed \textit{$k$-holed torus relations}, 
which represent the product of the right hand Dehn twists along the simple closed curves $\delta_i$
parallel to the boundary components of the $k$-holed torus $\Sigma_1^{k}$ 
as the product of the right hand Dehn twists along certain essential simple closed curves $\alpha_j$, 
in the form
$${\delta_1}{\delta_2} \cdots {\delta_k} = {\alpha_1} {\alpha_2} \cdots {\alpha_l} \qquad \text{in} \enskip \mathcal{M}_1^k. $$
For example, they started with the well-known $1$-holed torus relation
$$\delta_1 = (\alpha\beta)^6,$$
where the curves are as depicted in Figure \ref{fig:1-holedtorusrelation}.
By combining the $1$-holed torus relation with the lantern relation, they obtained the $2$-holed torus relation
$$ \delta_1\delta_2 = (\alpha_1\alpha_2\beta)^4,$$
which is also well-known,
where the curves are as depicted in Figure \ref{fig:2-holedtorusrelation}.
They successively constructed the $(k+1)$-holed torus relation by combining the 
$k$-holed torus relation with the \textit{lantern relation} (see, also~\cite{KO2})
until they obtained the $9$-holed torus relation.
However, for some reason, we will use the $7$-holed torus relation
$$ \delta_1\delta_2\delta_3\delta_4\delta_5\delta_6\delta_7 =
  \alpha_3\alpha_4\alpha_1\beta\sigma_5\alpha_2\beta_5\sigma_3\sigma_6
\alpha_6\beta_3\sigma_4, $$
where $\delta_1, \delta_2, \dots, \delta_7, \alpha_1, \alpha_2, \dots, \alpha_7, \beta, \sigma_3, \sigma_4, \sigma_5$ 
and $\sigma_6 $ are as depicted in Figure~\ref{fig:7-holedtorusrelation}, and $\beta_3=\alpha_3\beta\overline{\alpha}_3$ and $\beta_5=\alpha_5\beta\overline{\alpha}_5$. 

\begin{figure}[t]
 \begin{center}
  \includegraphics[keepaspectratio, width=194pt, clip]{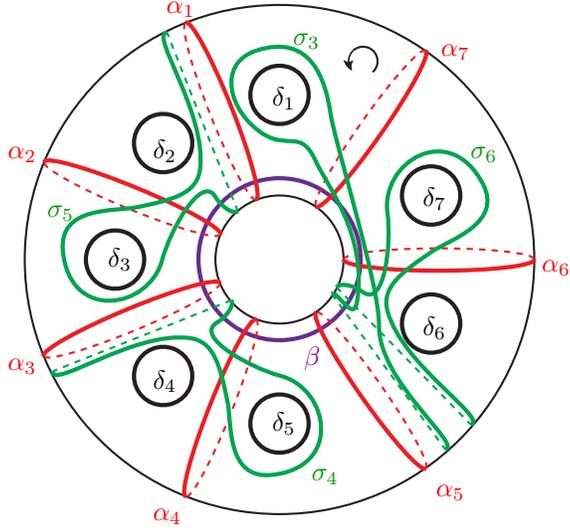}
  \put(-100,155){$\delta_1$} 
  \put(-145,135){$\delta_2$} 
  \put(-162,94){$\delta_3$} 
  \put(-145,48){$\delta_4$} 
  \put(-100,32){$\delta_5$} 
  \put(-43,68){$\delta_6$} 
  \put(-43,117){$\delta_7$}   
  \put(-140,190){\color{red}$\alpha_1$} 
  \put(-200,135){\color{red}$\alpha_2$}  
  \put(-200,55){\color{red}$\alpha_3$} 
  \put(-145,-2){\color{red}$\alpha_4$} 
  \put(-38,7){\color{red}$\alpha_5$} 
  \put(2,92){\color{red}$\alpha_6$} 
  \put(-36,174){\color{red}$\alpha_7$} 
  \put(-88,57){\color[rgb]{0.5,0,0.5}$\beta$}
  \put(-92,178){\color[rgb]{0,0.6,0}$\sigma_3$} 
  \put(-85,13){\color[rgb]{0,0.6,0}$\sigma_4$} 
  \put(-185,113){\color[rgb]{0,0.6,0}$\sigma_5$} 
  \put(-25,135){\color[rgb]{0,0.6,0}$\sigma_6$} 
  \caption{7-holed torus}
  \label{fig:7-holedtorusrelation}
 \end{center}
\end{figure}

\section{Proof of the upper bounds} \label{sec:proof}
Now, we prove Theorem~\ref{mainthm} by constructing Lefschetz fibrations with the claimed number of singular fibers.
Matsumoto's relation will be used for the proof of Theorem \ref{mainthm} $(1)$, and the $7$-holed torus relation will be 
used for $(2)$.
It is convenient to recall the following well-known lemma.

\begin{lem}  \label{lem:relmcg}  Let $\alpha$ and $\beta$ be simple closed curves in $\Sigma_g^k$. \\
\textup{(1)} If $\alpha \cap \beta = \emptyset$,  then $t_\alpha t_\beta = t_\beta t_\alpha$.\\
\textup{(2)} For any $\phi \in \mathcal{M}_g^k$, we have $t_{\phi(\alpha)} = \phi t_\alpha \phi^{-1} $.    \\
In particular, a conjugate of a right hand Dehn twist is also a right hand Dehn twist.
\end{lem}

\begin{proof}[Proof of Theorem~\textup{\ref{mainthm} (1)}]
We embed the surface $\Sigma_2^2$ into $\Sigma_{3+k}$ $(k \geq 0)$ as in Figure \ref{fig:embeddedmatsumotosrelation}.
\begin{figure}[t]
 \begin{center}
  \includegraphics[keepaspectratio, width=290pt, clip]{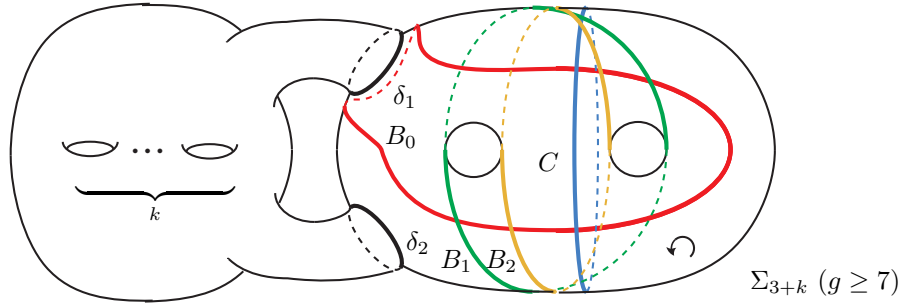}
  \put(-145,75){$\delta_1$} 
  \put(-140,20){$\delta_2$} 
  \put(-265,45){$\underbrace{\hspace{60pt}}_{k}$}
  \put(-148,60){$B_0$} 
  \put(-127,13){$B_1$} 
  \put(-110,13){$B_2$} 
  \put(-90,50){$C$} 
  \put(-10,5){$\Sigma_{3+k}$ ($g\geq7$)}
  \caption{Embedding $\Sigma_2^2$ in $\Sigma_{3+k}$}
  \label{fig:embeddedmatsumotosrelation}
 \end{center}
\end{figure}
Then, by Lemma \ref{lem:matsumotosrelation} we have 
\begin{align*}
\delta_1\delta_2 &= (B_0 B_1 B_2 C)^2 \\
&= B_0 B_1 B_2 C B_0 B_1 B_2 C \\
&= B_1 B_2 (\overline{B}_2 \overline{B}_1 B_0 B_1 B_2) C B_0 B_1 B_2 C \\
&= B_1 B_2 B_0^\prime C B_0 B_1 B_2 C,
\end{align*}
where $B_0^\prime = \overline{B}_2 \overline{B}_1 B_0 B_1 B_2$ is a right hand Dehn twist by 
Lemma~\ref{lem:relmcg}~$(2)$.
Multiplying $\overline{B}_2\overline{B}_1$ to the both sides, we obtain 
\begin{align*}
\overline{B}_2\overline{B}_1\delta_1\delta_2 =B_0^\prime C B_0 B_1 B_2 C. 
\end{align*}
Then the term on the left hand side is a commutator. 
To see this, by using Lemma~\ref{lem:relmcg}~$(1)$ we first rearrange it as
$$ \overline{B}_2\overline{B}_1\delta_1\delta_2
 =  (\overline{B}_2\delta_1)( \overline{B}_1\delta_2). $$
We observe that if one cuts $\Sigma_{3+k}$ along {$B_2$} and {$\delta_1$}, then the resulting surface is still connected. 
If one cuts $\Sigma_{3+k}$ along {$B_1$} and {$\delta_2$}, then the resulting surface is also connected 
(see Figure \ref{fig:cutcurve}). 
\begin{figure}[t]
 \begin{center}
  \includegraphics[keepaspectratio, width=338pt, clip]{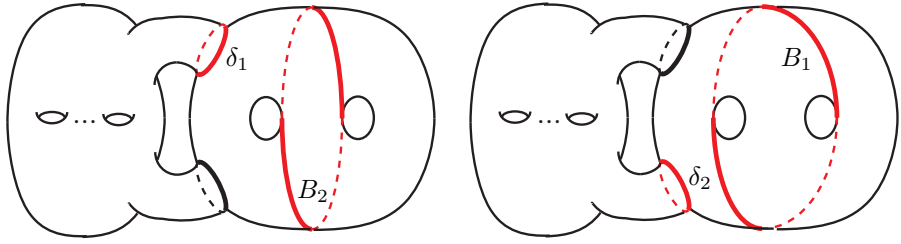}
  \put(-255,65){$\delta_1$} 
  \put(-228,15){$B_2$} 
  \put(-80,20){$\delta_2$} 
  \put(-45,65){$B_1$} 
  \caption{$\Sigma_{3+k} \setminus (B_2 \cup \delta_1)$ and $\Sigma_{3+k} \setminus (B_1 \cup \delta_2)$ are both connected.}
  \label{fig:cutcurve}
 \end{center}
\end{figure}
By the classification of surfaces, this observation implies that there exists an element 
$\phi \in \mathcal{M}_{3+k} $ such that
$\phi({\delta_1}) = {B_1}$ and $\phi({B_2}) = {\delta_2}$. 
Therefore, we have
\begin{align*}
(\overline{B}_2\delta_1)(\overline{B}_1\delta_2)  
&= { \overline{B}_2\delta_1} \overline{\phi({\delta_1})} \phi({B_2})  \\
&= \overline{B}_2\delta_1 \phi \overline{\delta}_1 \phi^{-1}\phi B_2\phi^{-1} \\
&= (\overline{B}_2\delta_1)\phi(\overline{\delta}_1 B_2)\phi^{-1} \\
&= [ \overline{B}_2\delta_1, \phi ] . 
\end{align*}
Consequently, we obtain
$$[ \overline{B}_2\delta_1, \phi ] = \underbrace{B_0^\prime C B_0 B_1 B_2 C}_6. $$
As we have mentioned before, this relation enables us to construct, for all $g = k + 3$ $\geq 3$, 
a genus $g$ Lefschetz fibration over the torus with six singular fibers, i.e., $N(g,1) \leq 6$ for all $g \geq 3$.  \qedhere
\end{proof}

\begin{proof}[Proof of Theorem~\textup{\ref{mainthm}~$(2)$}]
For the $7$-holed torus relation, the same procedure as above works well as follows.
(In fact, the $k$-holed torus relations ($2 \leq k \leq 7$) can be used similarly 
to prove $N(g,1) \leq 12-k$ for all $g \geq k$.)

Embed the $7$-holed torus $\Sigma_1^7$ into $\Sigma_{7+k}$ $(k \geq 0)$ as in Figure \ref{fig:embedded7holedtorusinsigma_g}.
\begin{figure}[t]
 \begin{center}
  \includegraphics[keepaspectratio, width=211pt, clip]{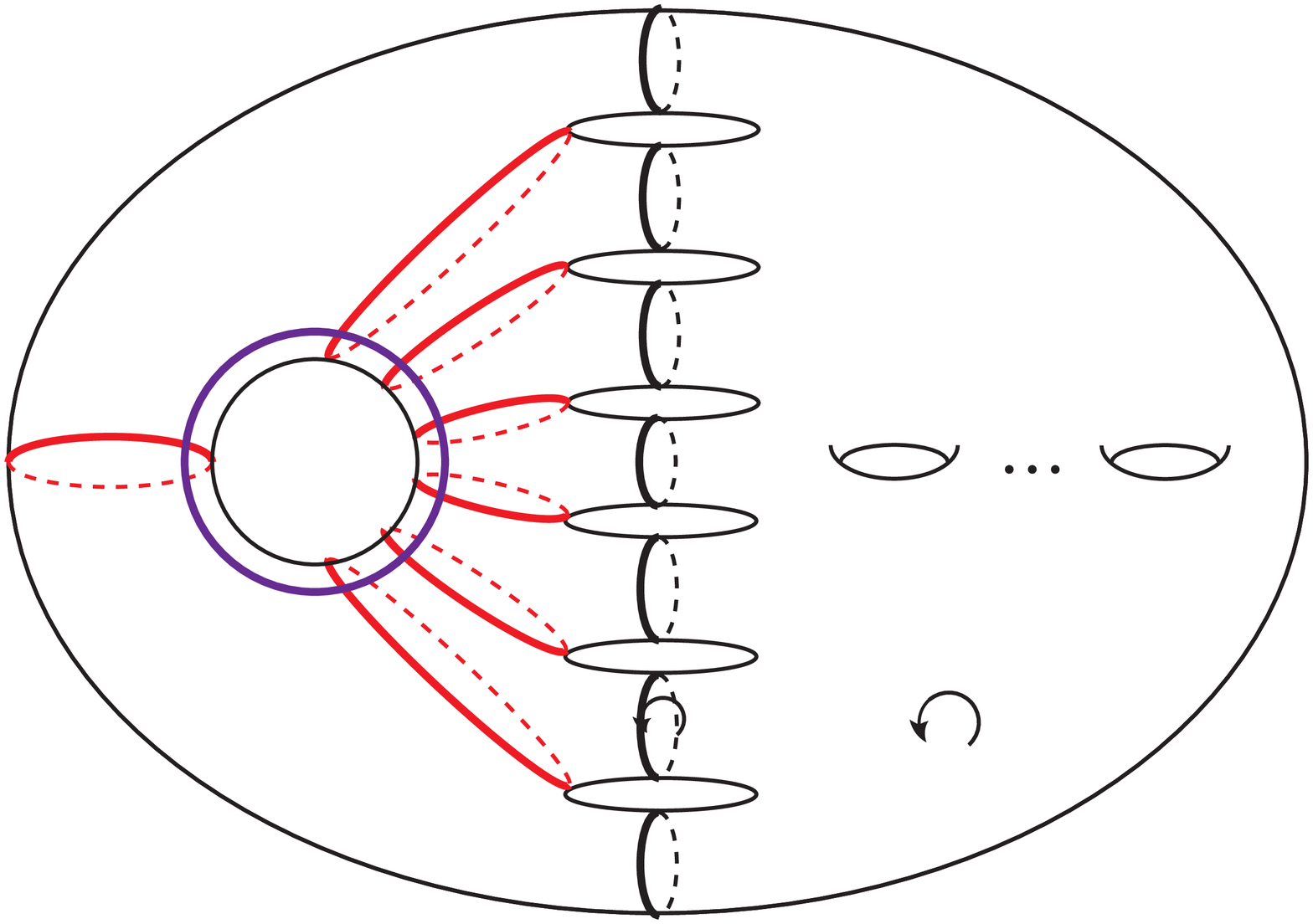}
  \put(-100,135){$\delta_1$} 
  \put(-100,115){$\delta_2$} 
  \put(-100,93){$\delta_3$} 
  \put(-100,72){$\delta_4$}  
  \put(-100,52){$\delta_5$}  
  \put(-100,30){$\delta_6$}  
  \put(-100,7){$\delta_7$}  
  \put(-130,132){$\alpha_1$}  
  \put(-130,110){$\alpha_2$}  
  \put(-130,88){$\alpha_3$}  
  \put(-130,56){$\alpha_4$}  
  \put(-130,36){$\alpha_5$}
  \put(-130,15){$\alpha_6$}  
  \put(-200,85){$\alpha_7$}
  \put(-170,100){$\beta$}      
  \put(-10,15){$\Sigma_{7+k} ~ (k \geq 0)$}
  \put(-75,65){$\underbrace{\hspace{60pt}}_{k}$}
  \caption{Embedding $\Sigma_1^7$ in $\Sigma_{7+k}$}
  \label{fig:embedded7holedtorusinsigma_g}
 \end{center}
\end{figure}
Then we rewrite the $7$-holed torus relation as 
\begin{align*}
\delta_1\delta_2\delta_3\delta_4\delta_5\delta_6\delta_7 
&= \alpha_3\alpha_4\alpha_1\beta\sigma_5\alpha_2\beta_5\sigma_3\sigma_6\alpha_6\beta_3\sigma_4 \\
&= \alpha_3\alpha_4\alpha_1\sigma_5\alpha_2\sigma_6\alpha_6
(\overline{\alpha}_6\overline{\sigma}_6\overline{\alpha}_2\overline{\sigma}_5)\beta
(\sigma_5\alpha_2\sigma_6\alpha_6) \\
& ~~~~~~~~~~~~~~
(\overline{\alpha}_6\overline{\sigma}_6)\beta_5(\sigma_6\alpha_6)
(\overline{\alpha}_6\overline{\sigma}_6)\sigma_3(\sigma_6\alpha_6)
\beta_3\sigma_4 \\
&= \alpha_3\alpha_4\alpha_1\sigma_5\alpha_2\sigma_6\alpha_6 
\beta^\prime \beta_5^\prime \sigma_3^\prime \beta_3\sigma_4,
\end{align*}
where $\beta^\prime = (\overline{\alpha}_6\overline{\sigma}_6\overline{\alpha}_2\overline{\sigma}_5)\beta
(\sigma_5\alpha_2\sigma_6\alpha_6)$, 
$\beta_5^\prime = (\overline{\alpha}_6\overline{\sigma}_6)\beta_5(\sigma_6\alpha_6)$ and
$\sigma_3^\prime = (\overline{\alpha}_6\overline{\sigma}_6)\sigma_3$
$(\sigma_6\alpha_6)$
are right hand Dehn twists.
Multiplying the inverse of $\alpha_3\alpha_4\alpha_1\sigma_5\alpha_2\sigma_6\alpha_6$ 
to the both sides, and rearranging the left hand side 
by using Lemma~\ref{lem:relmcg}~$(1)$, we obtain
$$  ( \overline{\alpha}_6\delta_1 \overline{\alpha}_1\delta_4 \overline{\alpha}_2\delta_5\overline{\alpha}_3)
( \delta_2\overline{\sigma}_5\delta_3\overline{\sigma}_6\delta_6\overline{\alpha}_4\delta_7 )
=\beta^{\prime}\beta_5^{\prime}\sigma_3^{\prime}\beta_3\sigma_4. $$
We observe that if one cuts $\Sigma_{7+k}$ along {$\alpha_6$}, {$\delta_1$}, {$\alpha_1$}, {$\delta_4$}, {$\alpha_2$}, {$\delta_5$} and {$\alpha_3$}, then the resulting surface is still connected. 
If one cuts $\Sigma_{7+k}$ along {$\delta_2$}, {$\sigma_5$}, {$\delta_3$}, {$\sigma_6$}, {$\delta_6$}, {$\alpha_4$} and {$\delta_7$}, then the resulting surface is also connected (see Figure \ref{fig:cutcurveof7holedtorus}). 
\begin{figure}[t]
 \begin{center}
  \includegraphics[keepaspectratio, width=337pt, clip]{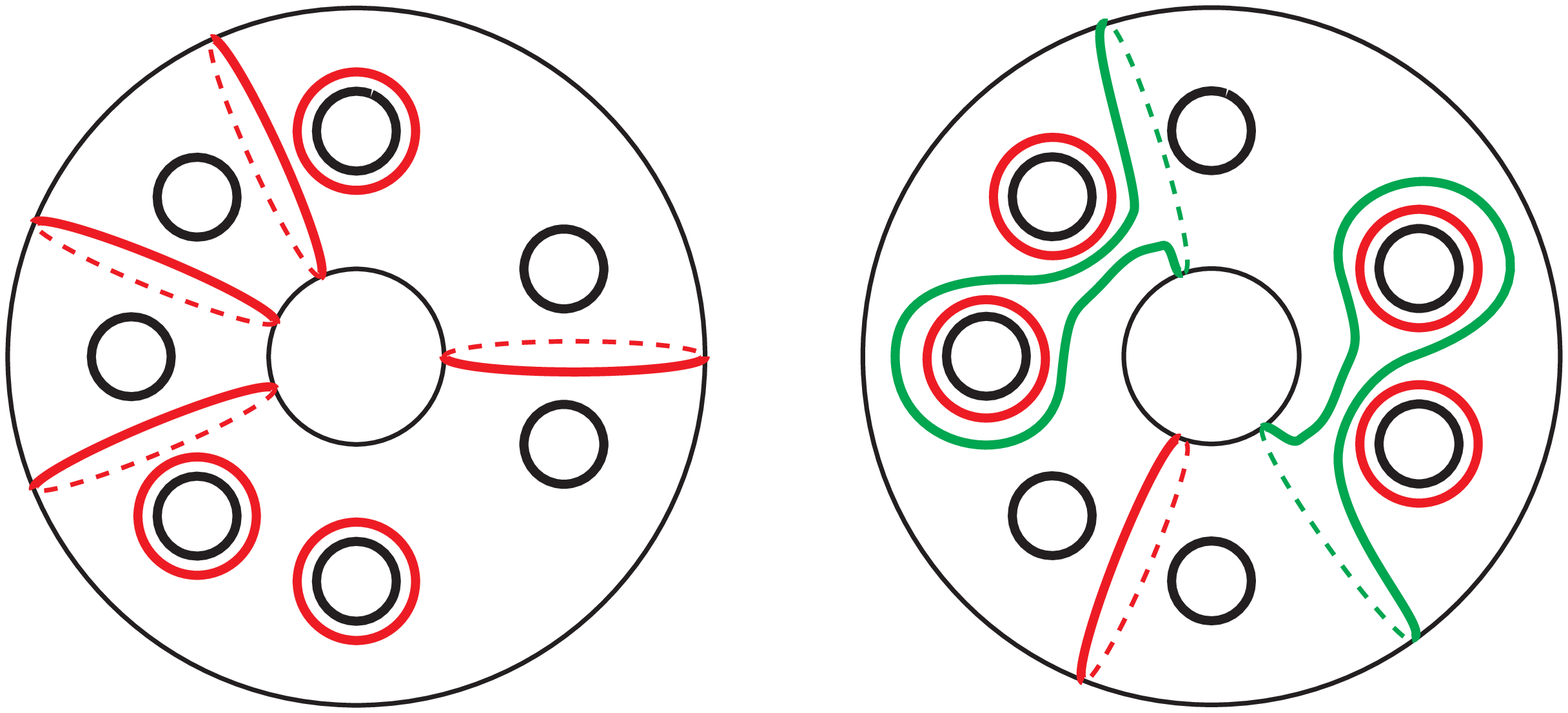}
  \put(-265,120){$\delta_1$} 
  \put(-300,107){$\delta_2$} 
  \put(-315,73){$\delta_3$} 
  \put(-300,40){$\delta_4$} 
  \put(-265,25){$\delta_5$} 
  \put(-220,53){$\delta_6$} 
  \put(-220,90){$\delta_7$} 
  \put(-300,150){$\alpha_1$}
  \put(-345,105){$\alpha_2$}
  \put(-343,40){$\alpha_3$}
  \put(-183,73){$\alpha_6$}
  \put(-80,120){$\delta_1$} 
  \put(-115,107){$\delta_2$} 
  \put(-130,73){$\delta_3$} 
  \put(-115,40){$\delta_4$} 
  \put(-80,25){$\delta_5$} 
  \put(-35,53){$\delta_6$} 
  \put(-35,90){$\delta_7$} 
  \put(-115,0){$\alpha_4$}
  \put(-135,50){$\sigma_5$}
  \put(-30,118){$\sigma_6$}
  \caption{$\Sigma_{7+k} \setminus (\alpha_6 \cup \delta_1 \cup \alpha_1 \cup \delta_4 \cup \alpha_2 \cup \delta_5 \cup \alpha_3)$ and 
$\Sigma_{7+k} \setminus (\delta_2 \cup \sigma_5 \cup \delta_3 \cup \sigma_6 \cup \delta_6 \cup \alpha_4 \cup \delta_7)$ are both connected. Note that the boundary components of $\Sigma_1^7$ are connected in $\Sigma_{7+k}$.}
  \label{fig:cutcurveof7holedtorus}
 \end{center}
\end{figure}
Therefore there exists an element $\phi \in \mathcal{M}_{7+k} $ such that 
$\phi({\alpha_3}) = {\delta_2}$, $\phi({\delta_5}) = {\sigma_5}$, $\phi({\alpha_2}) = {\delta_3}$, $\phi({\delta_4}) = {\sigma_6}$, $\phi({\alpha_1}) = {\delta_6}$, $\phi({\delta_1}) = {\alpha_4}$ and $\phi({\alpha_6}) = {\delta_7}$. 
Then we have
\begin{align*}
 & ( \overline{\alpha}_6\delta_1 \overline{\alpha}_1\delta_4 \overline{\alpha}_2\delta_5\overline{\alpha}_3) 
( \delta_2\overline{\sigma}_5\delta_3\overline{\sigma}_6\delta_6\overline{\alpha}_4\delta_7 ) \\
&= { \overline{\alpha}_6\delta_1 \overline{\alpha}_1\delta_4 \overline{\alpha}_2\delta_5\overline{\alpha}_3} 
\phi({\alpha_3}) \overline{\phi({\delta_5})}
\phi({\alpha_2}) \overline{\phi({\delta_4})}
\phi({\alpha_1}) \overline{\phi({\delta_1})}
\phi({\alpha_6}) \\
&=\overline{\alpha}_6\delta_1 \overline{\alpha}_1\delta_4 \overline{\alpha}_2\delta_5\overline{\alpha}_3 
\phi \alpha_3 \phi^{-1}
\phi \overline{\delta}_5 \phi^{-1}
\phi \alpha_2 \phi^{-1}
\phi \overline{\delta}_4 \phi^{-1}
\phi \alpha_1 \phi^{-1}
\phi \overline{\delta}_1 \phi^{-1}
\phi \alpha_6 \phi^{-1}
\\
&= (\overline{\alpha}_6\delta_1 \overline{\alpha}_1\delta_4 \overline{\alpha}_2\delta_5\overline{\alpha}_3)\phi ( \alpha_3\overline{\delta}_5\alpha_2\overline{\delta}_4\alpha_1\overline{\delta}_1\alpha_6)\phi^{-1} \\
&= [\overline{\alpha}_6\delta_1 \overline{\alpha}_1\delta_4 \overline{\alpha}_2\delta_5\overline{\alpha}_3, \phi ] . 
\end{align*}
Consequently, we obtain
$$[\overline{\alpha}_6\delta_1 \overline{\alpha}_1\delta_4 \overline{\alpha}_2\delta_5\overline{\alpha}_3, \phi ] = \underbrace{\beta^{\prime}\beta_5^{\prime}\sigma_3^{\prime}\beta_3\sigma_4}_{5}. $$
This relation implies the existence of a genus $g$ Lefschetz fibration over the torus with five singular fibers, 
for all $g \geq 7$, i.e., {$N(g,1) \leq 5$ for all $g \geq 7$}. \qedhere
\end{proof}

\begin{remark} 
$(1)$ If we could apply the same procedure as above to the $8$-holed torus relation or the $9$-holed torus relation,
then we would obtain $N(g,1) \leq 4$ for all $g \geq 8$ or $N(g,1) \leq 3$ for all $g \geq 9$, respectively.
However, it is impossible to do that, 
since the simple closed curves appearing in these relations are so complicated that we cannot choose 
two disjoint collections of these curves such that the surface obtained by cutting along 
each collection of curves is connected.\\
$(2)$ There is a possibility that the monodromy of Xiao's fibration~\cite{Xiao} can be used to prove $N(g,1) \leq 5$ for all $g \geq 3$.
As we mentioned in Section~\ref{sec:intro}, Xiao has discovered a genus $2$ Lefschetz fibration
over the sphere with seven singular fibers.
Although the exact monodromy of Xiao's fibration has not been known yet, 
by the existence of such a Lefschetz fibration we know that 
there exists a relation in $\mathcal{M}_2$ such as 
$$ c_1c_2c_3c_4c_5c_6c_7 = 1,$$
where the curves $c_i$ are essential simple closed curves in $\Sigma_2$.
Moreover, from the information of the abelianization of $\mathcal{M}_2$, 
we can deduce that four of $c_i$ are non-separating and the other three are separating.
In addition, let us assume that the relation can be lifted in $\mathcal{M}_2^2$ in a form
$$ \tilde{c}_1\tilde{c}_2\tilde{c}_3\tilde{c}_4\tilde{c}_5\tilde{c}_6\tilde{c}_7 = \delta_1\delta_2,$$
where the curves $\tilde{c}_i$ are simple closed curves in $\Sigma_2^2$ such that 
$P(\tilde{c}_i) = c_i$ for the natural homomorphism $P : \mathcal{M}_2^2 \rightarrow \mathcal{M}_2$, 
and $\delta_1$ and $\delta_2$ are the simple closed curves in $\Sigma_2^2$ 
parallel to the boundary components.
Then our technique used to prove Theorem~\ref{mainthm} can be applied to the above relation
by embedding $\Sigma_2^2$ into $\Sigma_{3+k}$ ($k\geq0$).
This argument would imply $N(g,1) \leq 5$ for all $g \geq 3$.
Note that the assumption of the existence of $\{\tilde{c}_i\}$ is equivalent to
the existence of a genus $2$ Lefschetz fibration over the sphere with seven singular fibers 
and two disjoint $(-1)$-sections (cf.~\cite{KO2}). 
\end{remark}

\end{document}